\documentclass[12pt]{article}

\usepackage{amsmath}
\usepackage{amsthm}
\usepackage{geometry}
\usepackage{mathrsfs}
\usepackage{amssymb}
\usepackage{amsfonts}
\usepackage[percent]{overpic}
\usepackage{bm}
\usepackage[all]{xy}
\usepackage{graphicx}
\usepackage{subfigure}
\usepackage{latexsym}
\usepackage[breaklinks=true,colorlinks, linkcolor=black, anchorcolor=blue, citecolor=blue]{hyperref}
\usepackage{epigraph}
\usepackage{hyperref}
\usepackage{fancyhdr}
\usepackage{comment}

\usepackage{mathtools}
\mathtoolsset{showonlyrefs}

\numberwithin{equation}{section}

\usepackage{color}

\newtheorem{definition}{Definition}[section]

\newtheorem{theorem}{Theorem}[section]
\newtheorem{que}{Question}[section]

\newtheorem{lemma}{Lemma}[section]
\newtheorem{example}{Example}[section]

\newtheorem{proposition}{Proposition}[section]

\theoremstyle{definition}
\newtheorem{remark}{Remark}[section]

\def\area{\operatorname{area}}

\def\sing{\operatorname{Sing}}

\def\meas{\operatorname{meas}}
\def\diam{\operatorname{diam}}
\def\dist{\operatorname{dist}}

\def\c{\operatorname{\mathbb C}}

\def\H{\operatorname{\mathcal{H}}}

\def\j{\operatorname{\mathcal{J}}}
\def\f{\operatorname{\mathcal{F}}}
\def\I{\operatorname{\mathcal{I}}}
\def\b{\operatorname{\mathcal{B}}}
\def\s{\operatorname{\mathcal{S}}}

\def\p{\operatorname{\mathcal{P}}}
\def\sing{\operatorname{Sing}}

\def\m{\operatorname{\mathcal{M}}}
\def\ct{\operatorname{Crit}}

\begin{document}
\title{Speiser meets Misiurewicz}
\author{Magnus Aspenberg and Weiwei Cui}

\date{}
\maketitle

\begin{abstract}
We propose a notion of Misiurewicz condition for transcendental entire functions and study perturbations of Speiser functions satisfying this condition in their parameter spaces (in the sense of Eremenko and Lyubich). We show that every Misiurewicz entire function can be approximated by hyperbolic maps in the same parameter space. Moreover, Misiurewicz functions are Lebesgue density points of hyperbolic maps if their Julia sets have zero Lebesgue measure. We also prove that the set of Misiurewicz Speiser functions has Lebesgue measure zero in the parameter space.


\end{abstract}

\section{Introduction and main results}

%Hyperbolic maps in transcendental dynamics have been understood quite well. 
A transcendental entire function $f$ is \emph{hyperbolic} if the post-singular set is a compact subset of the Fatou set. Such functions are uniformly expanding on their Julia sets with respect to certain conformal metric \cite{rempe8}. Allowing critical points in the Julia set will in general destroy uniform expansion. However, non-hyperbolic maps with certain conditions on the singular orbits  can still be understood to some extent. The main purpose of this paper is to explore a class of non-uniformly expanding maps which are in a sense singularly non-recurrent. This is the next natural step towards a better understanding of transcendental functions in the dynamical point of view.

Usually if we do not assume \emph{uniform hyperbolicity}, singular values are allowed in the Julia set. For non-uniformly hyperbolic maps, restrictions are often put on the singular orbits (orbits of singular values). %This line of research has been an intensive topic in rational dynamics over the past decades. 
Regarding this, there are several well known conditions, such as the Misiurewicz condition, semi-hyperbolicity, the Collet-Eckmann condition etc., in rational dynamics. Among them, the Misiurewicz condition, which roughly means that the critical points are non-recurrent, will be considered in this paper in the transcendental setting.

% We will in this paper propose a Misiurewicz condition in the transcendental setting.
A first attempt is to borrow the definition  from rational dynamics directly. Recall that a rational map satisfies the Misiurewicz condition if it is non-hyperbolic, has no parabolic cycles and moreover, the omega limit set of each critical point in the Julia set does not contain any critical point. However, a direct adoption of this definition does not sound reasonable. Transcendental functions satisfying the above condition may still behave quite differently. Intuitively, Misiurewicz maps should, in a sense, exhibit similar properties as hyperbolic maps. To present a reasonable definition for the Misiurewicz condition and to make perturbations meaningful, we probably want to consider those maps in better parameter spaces. Therefore, we propose such a definition in the following. To start with, we introduce several notions and notations.

\smallskip
Let $f$ be an entire function. The set of singularities of the inverse of $f$ consists of critical and asymptotic values in $\c$ of $f$, and is denoted by $\sing(f^{-1})$. We put, for some integer $k$,
$$\p(f):=\overline{\bigcup_{n\geq 0}f^{n}(\sing(f^{-1}))}~\text{~and~}~\p^k(f):=\overline{\bigcup_{n\geq k}f^{n}(\sing(f^{-1}))}.$$
The former is known as the \emph{post-singular set} of $f$. As usual, let $\f(f)$ and $\j(f)$ denote the Fatou and Julia sets of $f$ respectively. 

\begin{definition}\label{mis}
A  transcendental entire function $f$ satisfies the Misiurewicz condition if it is non-hyperbolic and satisfies the following conditions:
\begin{itemize}
\item[$(i)$] $\p(f)\cap\f(f)$ is compact, and 
\item[$(ii)$] there exists $k\in\mathbb{N}$ such that $\p^k(f)\cap\j(f)$ is hyperbolic.
\end{itemize}
\end{definition}

For simplicity, $f$ is called a Misiurewicz function if it satisfies the above Misiurewicz condition. Recall that a compact, forward invariant set $\Lambda$ is \emph{hyperbolic} (or is called a \emph{hyperbolic set}), if there exist $C>0$, $\lambda>1$ such that $
|(f^n)'(z)|\geq C\lambda^n$ for all $z\in\Lambda$ and any $n\geq 1$. The condition $(ii)$ means that every singular value in the Julia set will eventually fall into a hyperbolic set. (If $(ii)$ in the above definition is replaced by $\#(\p(f)\cap\j(f))<\infty$, we say that $f$ satisfies the \emph{Misiurewicz-Thurston} condition or is called a Misiurewicz-Thurston map.) It is not difficult to see that the Fatou sets of Misiurewicz functions can only have attracting cycles (see Proposition \ref{fatoucomponent} in the next section).

\begin{remark}
Stemming from \cite{misiurewicz1}, the Misiurewicz condition has been studied in various settings. For instance, in the (real and complex) quadratic family it means that the critical point is not recurrent, which, in turn, is equivalent to the well known semi-hyperbolicity condition introduced by Carleson, Jones and Yoccoz \cite{carleson1}. To extend the original definition of Misiurewicz to other settings, several (slightly) different definitions were considered which also motivates our definition; see, for instance, \cite{vanstrien1, graczyk1, aspenberg2}.
\end{remark}

With the above definition, our primary intention is to study perturbations of Misiurewicz entire functions. This is related to the famous \emph{hyperbolicity conjecture}, which states that hyperbolic maps are open and dense, or, equivalently, that every (non-hyperbolic) map can be approximated by hyperbolic maps. We will show that Misiurewicz functions have this approximation property, but we have to restrict to a suitable parameter space. It is clear by the definition above that Misiurewicz functions are post-singularly bounded and thus belong to the \emph{Eremenko-Lyubich class} 
$$\b:=\left\{f:\c\to\c\, \text{transcendental and entire,~}\, \sing(f^{-1})~\text{is bounded}\, \right\}.$$
(See \cite{sixsmith6} for a survey on the dynamics of this class.) However, such functions may have an infinite number of singular values and thus the  parameter space defined by Eremenko and Lyubich could be infinite-dimensional. Therefore, we consider the \emph{Speiser class}
$$\s:=\left\{f:\c\to\c\, \text{transcendental and entire,~}\, \sing(f^{-1})~\text{is finite}\, \right\}.$$
Functions in this class are often called \emph{Speiser functions} and have attracted a great deal of attention in recent years; see \cite{bishop1, bishop5} for instance. Naturally, it is expected that the hyperbolicity conjecture holds for Speiser functions.

To make hyperbolic approximation meaningful, we use parameter spaces constructed by Eremenko and Lyubich for Speiser functions; \cite[Section 3]{eremenko2}. More precisely, let $f,\,g\in\s$ be entire. We say that $f$ and $g$ are quasiconformally equivalent if there exist two quasiconformal maps $\varphi, \psi:\c\to\c$  such that $\varphi\circ f=g\circ\psi$. The parameter space $\m_f$ in which $f\in\s$ lies  is defined as the set of all functions $g$ which are quasiconformally equivalent to $f$. The parameter space $\m_f$ turns out to be a complex manifold of dimension $\#\sing(f^{-1})+2$.

\smallskip
Now our first result can be stated as follows.

\begin{theorem}\label{mishyp}
%Let $f$ be a Speiser function. Then any Misiurewicz map in $\m_f$ can be approximated by hyperbolic maps in $\m_f$.

Let $f\in\s$ be Misiurewicz. Then $f$ can be approximated by hyperbolic maps in $\m_f$.
\end{theorem}

Stronger result can be obtained provided that the Julia set has zero Lebesgue measure.

\begin{theorem}\label{denmis}
%Let $f$ be a Speiser function. Then any Misiurewicz map in $\m_f$ for which the Julia set has zero Lebesgue measure is a Lebesgue density point of hyperbolic maps in $\m_f$.

Let $f\in\s$ be Misiurewicz. Assume that $\j(f)$ has zero Lebesgue measure. Then $f$  is a Lebesgue density point of hyperbolic maps in $\m_f$.
\end{theorem}

Although Julia sets of rational Misiurewicz maps always have zero Lebesgue measure if the Fatou set is not empty, Julia sets of transcendental Misiurewicz maps could have positive Lebesgue measure in this case. Such an example will be given later on. We will show that hyperbolic maps have positive density around Misiurewicz maps if their Fatou sets are non-empty.

\medskip
In accordance with developments in the interval and rational dynamics (\cite{Sands1998,aspenberg2}), and as an application of our method of proof, we also show that Misiurewicz dynamics are rare in the following sense.

\begin{theorem}\label{main}
Let $f$ be a Speiser function. Then the set of Misiurewicz parameters in $\m_f$ has Lebesgue measure zero.
\end{theorem}

This result significantly generalizes a result of Bade\'nska in the exponential family \cite{badenska1} to the whole Speiser class. Our proof of the theorem also works for one parameter family of Speiser functions. More precisely, let $f\in\s$ and put
$$\f:=\{\lambda f:\, \lambda\in\c\setminus\{0\}\}.$$
Then our method shows that the set of Misiurewicz parameters in $\f$ has zero Lebesgue measure in the parameter plane $\c\setminus\{0\}$. In particular, this shows that the set of Misiurewicz parameters in the sine family $\{\lambda\sin(z): \lambda\in\c\setminus\{0\}\}$ has zero Lebesgue measure. However, if the singular values $\pm\lambda$ tend to $\infty$ under iteration (we call them escaping sine maps), then the set of such maps  has positive Lebesgue measure \cite{qiu2}. Note that escaping sine maps have no recurrent critical points. This also means that our suggested definition for the Misiurewicz condition is reasonable in this perspective.

\begin{remark}
Let $f$ be a Speiser function. The \emph{bifurcation locus}, denoted by $\b_f$, is a subset of $\m_f$, consisting of parameters which has at least one singular value that bifurcates. In other words, $g\in \b_f$ if there exists at least one singular value $s(g)$ such that $\{g^{n}(s(g))\}$ does not form a normal family in a neighborhood of $g$ in $\m_f$. Misiurewicz maps belong to the bifurcation locus. Although the above theorem tells that they have zero measure, but they are still dense in the bifurcation locus. Actually even post-singularly finite entire functions are already dense in the bifurcation locus by a classical normality argument similar as that of rational maps.
\end{remark}

\begin{remark}
In this paper we are only considering Misiurewicz functions in the setting of entire functions. The results of the paper extend also to meromorphic functions in the Speiser class.
\end{remark}

\bigskip
\noindent{\emph{Structure of the article.}} In Section $2$ we collect some known results of Speiser functions, in particular the definition of parameter spaces for such functions given by Eremenko and Lyubich. We also discuss a characterization of the Misiurewicz condition. Section $3$ is devoted to the proof of Theorem \ref{main}. The machinery developed in this section will be crucial for Theorems \ref{mishyp} and \ref{denmis}. Another key ingredient used in the proof of Theorem \ref{mishyp} is a version of the well known Wiman-Valiron theory developed by Bergweiler, Rippon and Stallard that can be applied to tracts without the need to have a globally defined entire function. This is presented briefly in Section $4$. In Section $5$ we prove Theorem \ref{denmis}. In the last section, we prove some result concerning Lebesgue measure of Julia sets for Misiurewicz Speiser functions and also give some explicit examples regarding this.

\section{Some preliminaries}

In this section we present some known results about Speiser functions. In particular, the construction of the complex structure on the parameter space $\m_f$ of $f\in\s$ is reviewed and restated briefly, following \cite{eremenko2}. A characterization of the Misiurewicz condition is also presented in terms of non-recurrence of singular values.

For the general theory of transcendental dynamics, we refer to \cite{bergweiler1}. In \cite{bergweiler6} some aspects of hyperbolic entire functions are considered. \cite{rempe16} gives a discussion on the notion of hyperbolicity in transcendental dynamics.

\subsection{Characterizing the Misiurewicz condition}
We discuss some properties of Misiurewicz entire functions and relate it to non-recurrence of singular values to critical points.

In the Definition \ref{mis}, the first condition implies that Misiurewicz functions cannot have parabolic cycles. As the definition also implies that Misiurewicz functions belong to the Eremenko-Lyubich class $\b$ (they are actually post-singularly bounded), we see immediately from a result of Eremenko and Lyubich that Misiurewicz functions do not have Baker domains \cite[Theorem 1]{eremenko2}. Since by $(ii)$, the (truncated) post-singular set (intersected with the Julia set) is a hyperbolic set, using a relation of the post-singular set with limit functions of wandering domains (e.g., \cite{bergweiler19}) we can also obtain that such functions do not have wandering domains. The expansion along the post-singular set also implies that Siegel disks do not exist for such functions. Therefore, we reach the following observation.

\begin{proposition}\label{fatoucomponent}
Let $f$ be an entire function satisfying Definition \ref{mis}. Then either $\j(f)=\c$ or $\f(f)$ consists of only attracting basins.
\end{proposition}

\begin{remark}
In this proposition, we did not restrict to Speiser functions. If only Speiser functions were considered, one can use a result of Eremenko and Lyubich (\cite[Theorem 3]{eremenko2}) and respectively Goldberg and Keen (\cite[Theorem 4.2]{goldberg4}) to exclude wandering domains directly.
\end{remark}

It is usually not easy to check whether a set is a hyperbolic set and thus not easy to check whether an entire function is Misiurewicz. In the rational setting, a criterion was given by a classical result of Ma\~n\'e in terms of non-recurrence of critical points \cite{mane4}. An analogous result for transcendental functions was proved by Rempe and van Strien \cite[Theorem 1.2]{rempe5}. In view of their result, it is plausible that the following is true. Let $\ct(f)$ be the set of critical points of $f$ and let $J_{crit}(f)=\ct(f)\cap\j(f)$, which is the set of critical points in the Julia set. 

\begin{que}\label{que}
Let $f$ be a non-hyperbolic entire function which is post-singularly bounded and has no parabolic cycles. Assume, furthermore, that $\p^k(f)\cap J_{crit}(f)=\emptyset$ for some $k\in\mathbb{N}$. Then $f$ satisfies the Misiurewicz condition.
\end{que}

\medskip

We will study perturbations of Misiurewicz entire functions within the \emph{Speiser class} $\s$. When restricted to this class, Question \ref{que} has a positive answer and thus gives a characterization of the Misiurewicz condition  in Definition \ref{mis}. For simplicity, we use $J_{sing}(f^{-1}):=\sing(f^{-1})\cap\j(f)$ for the set of singular values in the Julia set. 

\begin{definition}\label{mis'}
A Speiser entire function $f$ satisfies the Misiurewicz condition if it is non-hyperbolic, post-singularly bounded without parabolic cycles such that $\p^k(f)\cap J_{crit}(f)=\emptyset$ for some $k\in\mathbb{N}$.
\end{definition}

To see that these two definitions are equivalent, we first note that Definition \ref{mis} implies Definition \ref{mis'}. That the opposite is also true relies on the aforementioned result of Rempe and van Strien \cite[Theorem 1.2]{rempe5} which gives a criterion for a forward invariant set to be hyperbolic.

%It follows from Definition \ref{mis'} that for each $s\in\jing(f^{-1})$ there exists $k(s)\in\mathbb{N}$ such that $\p^{k(s)}(f,s)\cap\jrit(f)=\emptyset$. 
Put
\begin{equation}\label{trunca}
%\p_H(f):=\bigcup_{s\in\jing(f^{-1})}\p^{k(s)}(f,s).
\p_H(f):=\p^k(f)=\bigcup_{s\in J_{sing}(f^{-1})}\p^{k}(f,s).
\end{equation}
Note that this is a finite union as the function belongs to the class $\s$. Then we have the following result.
\begin{lemma}\label{hs}
Let $f\in\s$ be Misiurewicz. Then each $\p^{k}(f,s)$ is a hyperbolic set. So is $\p_H(f)$.
\end{lemma}

\begin{proof}
Let $f$ be such a function. Then by definition $\p_H(f)$ is compact,  forward invariant, and contains neither critical points nor parabolic periodic points. Moreover, $f\in\s$ implies that it does not have wandering domains by \cite{eremenko2} (see also \cite{goldberg4}). By \cite[Theorem 1.2]{rempe5} we can see that $\p_H(f)$ is a hyperbolic set.
\end{proof}

\subsection{Parameter space of Speiser functions}
Assume that $f\in\s$. Eremenko and Lyubich proved that the parameter space $\m_f$ can be endowed with a complex structure such that $\m_f$ is a complex manifold of (complex) dimension $q+2$, where $q$ is the number of singular values of $f$. The construction is as follows (we recast it briefly only for our purpose and refer to \cite[Section 3]{eremenko2} for more details): Let $\beta_1,\,\beta_2\in\c$ be distinct such that $f(\beta_1),\,f(\beta_2)\not\in\sing(f^{-1})$. Denote singular values of $f$ by $s_1(f), \dots, s_q(f)$ and put $s_{q+1}(f)=f(\beta_1)$ and $s_{q+2}(f)=f(\beta_2)$. Recall that $f,g\in\s$ are quasiconformally equivalent to $f$, if there exist two quasiconformal maps $\varphi,\psi: \c\to\c$ such that $\varphi\circ f=g\circ\psi$. Define $\m_f(\beta_1,\beta_2)$ to be the set of functions $g$ such that $g$ is quasiconformally equivalent to $f$ and $\psi$ fixes $\beta_1$ and $\beta_2$. ($\m_f(\beta_1,\beta_2)$ thus defined implies that if $g$ is conformally conjugate to $f$ then $g=f$.) Then 
$$\m_f=\bigcup_{\beta_1,\beta_2}\m_{f}(\beta_1,\beta_2).$$
This is called the \emph{parameter space} of $f$.

\smallskip
With $\Psi(f):=(s_1(f),\dots,s_{q+2}(f))$, $(\m_{f}(\beta_1,\beta_2),\Psi )$ becomes a chart and the collection $\{(\m_{f}(\beta_1,\beta_2),\Psi )\}$ defines a $(q+2)$-dimensional complex atlas on $\m_f$ which makes $\m_f$ a complex manifold. We will always fix this chart on $\m_f$ and say that a subset $A\subset \m_f$ has Lebesgue measure zero if $\Psi(A\cap \m_{f}(\beta_1,\beta_2))\subset \c^{q+2}$ has $(2q+4)$-dimensional Lebesgue measure zero. This is the Lebesgue measure used in Theorem \ref{main}.

\medskip

\noindent{\bf{\emph{Notations.}}} We collect here some notations which will be used throughout the paper. $B(0,r)$ denotes $(q+2)$ complex dimensional ball in the parameter space $\m_f$ for $f\in\s$ with $q$ singular values. While by $D(a,r)$ we always mean the one complex dimensional disk with center $a$ and radius $r$. We will only use Euclidean metric unless otherwise mentioned. For a holomorphic function $f$, the $n$-th derivative of $f$ at a point $z$ is denoted by $D^n f(z)$. Sometimes $f'(z)$ is used as the same meaning for $Df(z)$. By $A\sim_C B$ we mean $\frac{1}{C} B \leq A \leq C B$. The notation $\meas(A)$ denotes the two-dimensional Lebesgue measure of a set $A\subset\c$. We also use $A(r, R)$ for the annulus centred at the origin with inner and outer radius $r$ and $R$, respectively.

\section{Speiser Misiurewicz maps are rare}

In this section, we will prove Theorem \ref{main} stated in the introduction. We also note that Misiurewicz functions mentioned in the sequel will always mean Speiser functions satisfying the Misiurewicz condition proposed in Definition \ref{mis}, unless otherwise stated. The proof uses absence of invariant line fields for Misiurewicz functions, a phase-parameter relation and certain version of distortion control along the post-singular set.

\begin{comment}
\begin{theorem}\label{mtm}
Let $f$ be a Speiser function. Then the set of Misiurewicz parameters in the sense of Definition \ref{mis'} in the parameter space $\m_f$ has Lebesgue measure zero.
\end{theorem}
\end{comment}

\medskip
As Misiurewicz entire functions are post-singularly bounded, we will thus assume that for such a given function $f_0$, the singular set $\sing(f_{0}^{-1})\subset D(0,R)$ for some $R>0$. By continuity, this can also be made for parameters arbitrarily close to $f_0$. Let $B(0,r)$ be a full dimensional parameter ball in $\m_{f_0}$. So for all $a\in B(0,r)$ and for all $z\in D(0,R)$, there exists a constant $B>0$ such that
\begin{equation}\label{defR}
\sup_{a\in B(0,r), z\in D(0,R)}|D^2 f_a(z)|\leq B.
\end{equation}

\subsection{Phase-parameter relation}
Let $f_0\in\s$ be a Misiurewicz entire function. We will only focus on singular values of $f_0$ in the Julia set in the sequel, which are denoted by $s_j(0)$. By definition, there exists an integer $k$ such that the forward orbit of 
\begin{equation}\label{kj}
v_j(0):=f_{0}^{k}(s_j(0))
\end{equation}
does not contain any critical points. So $v_j(0)$ belongs to a hyperbolic set as ensured by Lemma \ref{hs}. Define
$$v_j(a)=f_{a}^{k}(s_j(a)),$$
where $s_j(a)$ is the corresponding singular value of $f_a$ to $s_j(0)$. By Eremenko and Lyubich's construction of parameter spaces for Speiser functions \cite{eremenko2}, $s_j(a)$ is holomorphic in $a$ for all $j$. We shall thus consider the evolution of the following sequence of maps
\begin{equation}\label{pp}
\xi_{n,j}(a):=f_{a}^{n}(v_j(a)),
\end{equation}
which relates parameter space to the phase space and we will show that it has strong distortion in small one-dimensional disks around $f_0$. 

For simplicity, we write
\begin{equation}\label{hyset}
\Lambda_0=\p_H(f_0).
\end{equation}
By Lemma \ref{hs}, $\Lambda_0$ is a hyperbolic set. Let $B(0,r)$ be a small ball centred around $f_0$ of (complex) dimension $q+2$, where $q=\#\sing(f_{0}^{-1})$.
So there is a holomorphic motion (\cite[page 229]{shishikura8})
$$h: B(0,r)\times\Lambda_0\to\c$$
such that
\begin{itemize}
\item for each $a\in B(0,r)$, $h(a,z)$ is injective in $z$;
\item for each $z\in\Lambda_0$, $h(a,z)$ is holomorphic in $a$;
\item $h_a\circ f_0=f_a\circ h_a$ holds for all $z\in\Lambda_0$.
\end{itemize}
Denote $\Lambda_a=h_a(\Lambda_0)$. We will consider the functions
\begin{equation}\label{xj}
x_j(a)=v_j(a)-h_{a}(v_j(0)).
\end{equation}
$x_j(a)$ is said to be \emph{transversal} in some set if $x_j(a) \not\equiv 0$ in this set.

\smallskip

The following lemma is a special case of \cite[Theorem 1.1]{rempe5}, which tells us that Misiurewicz maps are unstable.

\begin{lemma}\label{nilf}
Let $f_0\in\s$ be Misiurewicz in the sense of Definition \ref{mis'}. Then the Julia set $\j(f_0)$ supports no invariant line fields.
\end{lemma}

 Let $D(0,r)\subset B(0,r)$ be a one-dimensional disk in the parameter space $\m_f$. Since $\m_f$ is a complex manifold, each such disk can be determined by a direction $v\in\mathbb{P}(\c^{q+1})$ such that the plane for which $D(0,r)$ belongs can be parameterized by $\{a\in\c: a=t\alpha, t\in\c\}$ with $\alpha=\{\alpha_1,\dots,\alpha_{q+2}\}$ being a representative for $\mathbb{P}(\c^{q+1})$. We shall prove a version of Theorem \ref{main} for one-dimensional disks around $f$ for almost every directions. Then Theorem \ref{main} follows from this and Fubini's theorem.

Now, if the functions $x_j$ are  identically equal to zero in $B(0,r)$ for all $j$, then the maps in $B(0,r)$ are structurally stable. This means that there is a structurally stable component consisting of Misiurewicz maps, which contradicts with Lemma \ref{nilf} by the theory of Ma\~n\'e, Sad and Sullivan \cite{eremenko2, mane2} (also \cite[Theorem 4.2]{mcmullen4}). So we have the following.

\begin{lemma}[Transversality]\label{indi}
For $a\in B(0,r)$ with $f_0$ being Misiurewicz, there is at least one $j$ such that $x_j(a)$ is not identically equal to zero.
\end{lemma}

Define
\begin{equation}\label{defx}
\I:=\left\{\,j:\, x_j\not\equiv 0\,\right\}.
\end{equation}
Lemma \ref{indi} tells that $\I\neq\emptyset$ in $B(0,r)$. So $\{a\in B(0,r): x_j(a)=0\}$ is an analytic set of codimension one for each $j\in\I$. Therefore, for almost every direction the function $x_j$ for $j\in\I$ is not identically equal to zero in the disk $D(0,r)$ which is determined by this direction.

\smallskip
Now we can conclude that, for some $A>0$, the function $x_j$ for $j\in\I$ can be written as
\begin{equation}\label{cex}
x_{j}(a)=A a^k + \mathcal{O}(a^{k+1})
\end{equation}
for all $a$ sufficiently close to the parameter $0$, where $k$ is some integer. (It is plausible that $k=1$ in \eqref{cex}, but we cannot prove this here.) By \eqref{cex} we then have
\begin{equation}\label{cx}
\left|\frac{x_j(a)}{x'_j(a)} \right|\sim |a|
\end{equation}
holds for all $a\in D(0,r)$ and for small $r$.

\smallskip
Now by taking $r$ sufficiently small, we can make sure that for all $a\in B(0,r)$ singular orbits of $f_a$ follow the corresponding singular orbits of $f_0$ initially. In other words, they stay in a small neighbourhood of $\Lambda_0$ for the first iterates. However, due to the expansion of $f_0$ on $\Lambda_0$ it is highly possible that small disks around a singular value will grow exponentially. To phrase differently, the small one-dimensional disk $D(0,r)$ containing the starting function will grow to some large scale under the action of $\xi_{n,j}$. To make this more precise, we first note that $f_a$ will inherit uniform expansion from that of $\Lambda_0$, as long as $f_a(z)$ stays close to the hyperbolic set $\Lambda_0$.

The following is a simple lemma which can be deduced from hyperbolicity of $\Lambda_0$.

\begin{lemma}
Let $f_0$ be Misiurewicz and $\Lambda_0$ be defined as in \eqref{hyset}. Then there is a neighbourhood $\mathcal{O}$ of $\Lambda_0$, $C>0$ and $\lambda_0>1$ such that $|Df_{0}^{n}(z)|\geq C\lambda_{0}^n$ if $f_{0}^{j}(z)\in\mathcal{O}$ for all $j\leq n$.
\end{lemma}

This immediately implies the following.

\begin{lemma}\label{lem34}
There exist $r_0>0$ and $\lambda>1$ such that $|Df^{j}_a(z)|\geq C\lambda^j$ for all $a\in B(0,r)$ with all $r\leq r_0$ and $C>0$, provided that $f_{a}^{k}(z)\in\mathcal{O}$ for all $k\leq j$.
\end{lemma}

This means that two points in $\mathcal{O}$ will separate as long as their forward orbits stay in $\mathcal{O}$. More precisely, if $f_{a}^{j}(z),\,f_{a}^{j}(w)\in\mathcal{O}$ for all $j\leq n$, one has
\begin{equation}\label{repeleach}
|f_{a}^{n}(z)-f_{a}^{n}(w)|\geq \lambda^n |z-w|.
\end{equation}

We will do analysis in this neighbourhood $\mathcal{O}$, where any map near our starting map $f_0$ inherits expansion. With this expansion we show later on that a small parameter disk will grow to a definite size (``large scale''). We will take sufficiently small $\delta'>0$ such that the $\delta'$-neighbourhood of $\Lambda_a$ is well inside $\mathcal{O}$. In other words, $\delta'$ is chosen small enough such that $\{z: \dist(z,\Lambda_a)\leq \delta'\}\subset \mathcal{O}$.

\subsection{Distortions along the post-singular set}

Recall that $\xi_{n,j}(a)=f_{a}^{n}(v_j(a))$. We define 
$$\mu_{n,j}(a)=h_a(\xi_{n,j}(0)),$$
where $h_a$ is the holomorphic motion defined above. In the following comparison of phase derivatives, we shall use a simple lemma stated now.

\begin{lemma}\label{simplemma}
Let $b_k\in\c$ for $1\leq k\leq n$. Then
$$\left|\prod_{k=1}^{n} (1+b_k) - 1\right|\leq \exp\left(\sum_{k=1}^{n}|b_k| \right) - 1.$$
\end{lemma}

\begin{lemma}\label{sdort}
For any $\varepsilon>0$ sufficiently small, there exists $\delta'>0$ sufficiently small and $r>0$ such that for all $a\in B(0,r)$ we have
\begin{equation}
\left|\frac{Df^{n}_{a}(\mu_{0,j}(a))}{Df_{a}^{n}(\xi_{0,j}(a))} - 1 \right|<\varepsilon,
\end{equation}
provided that $|\mu_{k,j}(a)-\xi_{k,j}(a)|<\delta'$ holds for all $a$ and all $k\leq n$.
\end{lemma}

\begin{proof}
Using the chain rule and the fact that $\xi_{k,j}(a)=f_{a}^{k}(\xi_{0,j}(a))$ and $\mu_{k,j}(a)=f_{a}^{k}(\mu_{0,j}(a))$, one can obtain, by Lemma \ref{simplemma}, that
\begin{equation}
\begin{aligned}
\left|\frac{Df^{n}_{a}(\mu_{0,j}(a))}{Df_{a}^{n}(\xi_{0,j}(a))} - 1 \right|&=\left|\prod_{k=0}^{n-1}\frac{Df_{a}(\mu_{k,j}(a))}{Df_{a}(\xi_{k,j}(a))} - 1 \right|\\
&\leq \exp\left(\sum_{k=0}^{n-1}\left|\frac{Df_{a}(\mu_{k,j}(a))}{Df_{a}(\xi_{k,j}(a))} - 1\right| \right)-1.
\end{aligned}
\end{equation}

Thus to obtain the result it suffices to prove that the sum appearing on the right-hand side in the above inequality can be made arbitrarily small. We have, by Lemma \ref{lem34},
\begin{equation}
\begin{aligned}
\sum_{k=0}^{n-1}\left|\frac{Df_{a}(\mu_{k,j}(a))}{Df_{a}(\xi_{k,j}(a))} - 1\right|&\leq\sum_{k=0}^{n-1}\left|\frac{Df_{a}(\mu_{k,j}(a))-Df_{a}(\xi_{k,j}(a))}{Df_{a}(\xi_{k,j}(a))}\right|\\
&\leq \frac{1}{C\lambda} \sum_{k=0}^{n-1}\left|Df_{a}(\mu_{k,j}(a))-Df_{a}(\xi_{k,j}(a))\right|\\
&\leq \frac{1}{C\lambda}\max_{z\in D(0,R), a\in B(0,r)} |D^2 f_{a}(z)|\sum_{k=0}^{n-1}\left|\mu_{k,j}(a)-\xi_{k,j}(a)\right|\\
&\leq \frac{B}{C\lambda} \sum_{k=0}^{n-1}\frac{1}{\lambda^{n-k}}\left|\mu_{n,j}(a)-\xi_{n,j}(a)\right|\\
&\leq C_{\lambda}\delta',
\end{aligned}
\end{equation}
where $C_{\lambda}$ depends only on $\lambda$ and $B$. By choosing $\delta'>0$ sufficiently small we can reach our conclusion.
\end{proof}

The above lemma gives strong distortions up to some scale for all parameters as along as two orbits in the dynamical plane are both close to the singular orbits. It means that, under the assumption of the above lemma, for some constant $C>1$,
\begin{equation}
\begin{aligned}
\left|\xi_{n,j}(a)-\mu_{n,j}(a)\right|&\sim_C |Df_{a}^{n}(\xi_{0,j}(a))||\xi_{0,j}(a)-\mu_{0,j}(a)|\\
&=|Df_{a}^{n}(\xi_{0,j}(a))||x_{j}(a)|
\end{aligned}
\end{equation}
To show that the function $\xi_{n,j}$ is an almost linear map, we shall also need to compare phase and parameter derivatives. %This essentially relies on the instability of Misiurewicz parameters.

First, under the assumption of the above lemma, we have, by expanding the function $f_{a}^n$ near $\xi_{0,j}(a)$, that
\begin{equation}\label{express}
\xi_{n,j}(a)-\mu_{n,j}(a)=Df_{a}^{n}(\xi_{0,j}(a))x_{j}(a)+ E_{n,j}(a),
\end{equation}
where $E_{n,j}$ satisfies $|E_{n,j}(a)|\leq C_1 |\xi_{n,j}(a)-\mu_{n,j}(a)|$ for some small constant $C_1$.

\begin{lemma}[Comparison of phase and parameter derivatives]\label{pp}
For any small $\varepsilon>0$, by choosing $\delta'>0$ sufficiently small we have that, for any $\delta''<\delta'$ there exists $r>0$ with the following property: For any $a\in D(0,r)$, if $|\xi_{k,j}(a)-\mu_{k,j}(a)|\leq\delta'$ for $k\leq n$ and $|\xi_{n,j}(a)-\mu_{n,j}(a)|\geq\delta''$, then
\begin{equation}
\left|\frac{\xi'_{n,j}(a)}{Df_{a}^{n}(\xi_{0,j}(a))x'_j(a)} -1 \right|<\varepsilon.
\end{equation}
\end{lemma}

\begin{proof}
Denote by $F_{n,j}(a)=Df_{a}^{n}(\xi_{0,j}(a))$. By the discussion after Lemma \ref{sdort} we have
$$|\xi_{n,j}(a)-\mu_{n,j}(a)|\sim_C |F_{n,j}(a)||x_{j}(a)|.$$
So, by the assumption of the lemma, for sufficiently small $r>0$ there exists $\varepsilon'>0$ such that
\begin{equation}\label{sg}
(1-\varepsilon')\delta''\leq|F_{n,j}(a)||x_{j}(a)|\leq (1+\varepsilon')\delta'.
\end{equation}
Then we see that
$$|F_{n,j}(a)|^{-1}(1-\varepsilon')\delta''\leq|x_{j}(a)|\leq |F_{n,j}(a)|^{-1}(1+\varepsilon')\delta'.$$
By using the fact that $|F_{n,j}(a)|\geq C \lambda^{n}$ and by taking logarithm on both sides of the above formula, we have
\begin{equation}\label{imn}
\log\frac{C\lambda^n}{(1+\varepsilon')\delta'}\leq\log\frac{|F_{n,j}(a)|}{(1+\varepsilon')\delta'}\leq -\log|x_{j}(a)|\leq\log\frac{|F_{n,j}(a)|}{(1-\varepsilon')\delta''}.
\end{equation}

Now it follows from \eqref{express} that
\begin{equation}
\xi'_{n,j}(a)-\mu'_{n,j}(a)=F'_{n,j}(a)x_{j}(a)+F_{n,j}(a)x'_{j}(a)+E'_{n,j}(a),
\end{equation}
which can be written as
\begin{equation}\label{imm}
\xi'_{n,j}(a)=F_{n,j}(a)x'_{j}(a)+F_{n,j}(a)\left(\frac{F'_{n,j}(a)x_{j}(a)}{F_{n,j}(a)}+\frac{E'_{n,j}(a)+\mu'_{n,j}(a)}{F_{n,j}(a)}\right).
\end{equation}
Since $|F_{n,j}(a)|=|Df_{a}^{n}(\xi_{0,j}(a))|=\prod_{k=0}^{n-1}|Df_{a}(\xi_{k,j}(a))|$, we see that, combined with \eqref{imn},
\begin{equation}
\begin{aligned}
\frac{|F'_{n,j}(a)|}{|F_{n,j}(a)|}=\sum_{k=0}^{n-1}\frac{|D_a Df_{a}(\xi_{k,j}(a))|}{|Df_{a}(\xi_{k,j}(a))|}&\leq B\sum_{k=0}^{n-1}\frac{1}{|Df_{a}(\xi_{k,j}(a))|}\\
&\leq \frac{B}{C\lambda} n\\
&\lesssim \frac{B}{C\lambda}\left(\log\frac{\lambda}{(1+\varepsilon')\delta'}\right)^{-1}(-\log|x_{j}(a)|).
\end{aligned}
\end{equation}
This implies that, if $a$ tends to $0$, we have, by \eqref{cex},
$$\frac{|F'_{n,j}(a)x_{j}(a)|}{|F_{n,j}(a)|}\to 0.$$
Moreover, since both $|E'_{n,j}(a)|$ and $|\mu'_{n,j}(a)|$ are uniformly bounded above for all $a\in D(0,r)$, by using \eqref{sg} and \eqref{cx} we also obtain that
$$\frac{|E'_{n,j}(a)+\mu'_{n,j}(a)|}{|F_{n,j}(a)||x'_{j}(a)|}=\frac{|E'_{n,j}(a)+\mu'_{n,j}(a)|}{|F_{n,j}(a)||x_{j}(a)|}\frac{|x_j(a)|}{|x'_{j}(a)|}\to 0.$$
The above discussions altogether give us the desired conclusion.
\end{proof}

We also need the following result.

\begin{lemma}\label{sdort2}
Assume that $a,b\in B(0,r)$ satisfy the conditions of Lemma \ref{sdort}. Then the following is true:
\begin{equation}
\begin{aligned}
\left|\frac {Df_{a}^{n}(v_j(a))}{Df_{b}^{n}(v_j(b))} -1 \right|<\varepsilon.
\end{aligned}
\end{equation}
\end{lemma}

\begin{proof}
We note that $f'_{a}(\xi_{m,j}(a))$ depends analytically on $a$ for $a\in B(0,r)$. So we can write
\begin{equation}\label{exphere}
f'_a(\xi_{m,j}(a))=f'_0(\xi_{m,j}(0))\left(1+ b_{\ell} a^{\ell} +\mathcal{O}\left(a^{\ell+1}\right) \right)
\end{equation}
for some constant $b_{\ell}$ and integer $\ell\geq 1$.
Moreover, by assumption we have
\begin{equation}
\begin{aligned}
\delta'>|\mu_{n,j}(a)-\xi_{n,j}(a)|&=|f_{a}^{n}(\xi_{0,j}(a))-f_{a}^{n}(\mu_{0,j}(a))|\\
&\geq C\lambda^{n}|\xi_{0,j}(a)-\mu_{0,j}(a)|\\
&=C\lambda^n |x_j(a)|.
\end{aligned}
\end{equation}
So there exists a constant $C'$ depending on $\lambda$ and $\delta'$ such that
$$n\leq C'\log\frac{1}{|x_j(a)|}.$$
Together with \eqref{cex}, this gives, for some constant $C''$ (depending on $C'$ and the integer $k$ in \eqref{cex}), that
$$n\leq C''\log\frac{1}{|a|}.$$

Now by the chain rule and \eqref{exphere} we have
\begin{equation}
\begin{aligned}
\frac {Df_{a}^{n}(v_j(a))}{Df_{b}^{n}(v_j(b))} = \prod_{m=0}^{n-1}\frac{f'_{a}(\xi_{m,j}(a))}{f'_{b}(\xi_{m,j}(b))}&=\prod_{m=0}^{n-1}\frac{1+b_{\ell} a^{\ell}+ \mathcal{O}\left(a^{\ell+1}\right)}{1+b_{\ell} b^{\ell}+ \mathcal{O}\left(b^{\ell+1}\right)}\\
&=\frac{1+cn a^{\ell}+\cdots}{1+cn b^{\ell}+\cdots},
\end{aligned}
\end{equation}
where $c$ is a constant depending on $b_{\ell}$. So by taking $r>0$ sufficiently small, this product can be made sufficiently close to $1$. This gives the lemma.
\end{proof}

Lemma \ref{sdort2}, together with Lemma \ref{pp}, tells us that a small disk will finally grow to a definite size ("large scale") with strong bounded distortion. We say that $D(a_0, r_0)$ with $a_0\in D(0,r)$ is a $\kappa$-Whitney disk if there exists $\kappa\in (0,1)$ such that $r_0/|a_0|=\kappa$.

\begin{lemma}\label{est}
Let $f_0$ be Misiurewicz. There exists some $\kappa\in (0,1)$ such that for all $r>0$ sufficiently small and any $\varepsilon>0$ small the following holds. For any $\kappa$-Whitney disk $D\subset D(0,r)$ we have
\begin{equation}\label{stdi}
\left|\frac{\xi'_{n,j}(a)}{\xi'_{n,j}(b)} - 1 \right|<\varepsilon
\end{equation}
as long as conditions in Lemma \ref{pp} are satisfied for all $a,b\in D$.
\end{lemma}

\begin{proof}
Since $D\subset D(0,r)$ is a Whitney disk, there exists a constant $\kappa'$ depending on $\kappa$ such that
$$\dist(0,D)\leq \kappa' \diam(D).$$
This means that the function $x_j$ has bounded distortion in $D$. Then by using Lemma \ref{sdort}, Lemma \ref{pp} and Lemma \ref{sdort2}, we see that \eqref{stdi} holds for all parameters in the Whitney disk $D$.
\end{proof}

\begin{lemma}\label{scalelarge}
Let $f_0$ be Misiurewicz. For each $j\in\I$ there exist some $\kappa\in (0,1)$ and $S>0$ such that for all $r>0$ sufficiently small, one can find $n_j$ such that $\xi_{n_j,j}(D(a_0, r_0))\subset\mathcal{O}$ has diameter at least $S$, where $D(a_0, r_0)$ is any $\kappa$-Whitney disk.
\end{lemma}

\begin{proof}
By Lemma \ref{est} there exists $\kappa$ such that $\xi_{n_j,j}$ has strong distortion \eqref{stdi} up to time $n_j$ for all parameters in any $\kappa$-Whitney disk. So we see by \eqref{stdi} and Lemma \ref{sdort} that there is a constant $C_0>0$ such that
$$\diam\xi_{n_j,j}(D(a_0, r_0))\geq |\xi'_{n_j,j}(a_0)|r_0\geq C_0 \left|Df_{a_0}^{n_j}(v_j(a_0))\right||x'_j(a_0)|r_0.$$
As $x_j$ has bounded distortion in Whitney disks and the derivative of $f^{n_j}_{a_0}$ grows exponentially as long as it stays within $\mathcal{O}$, we can find a number $S$, say $S=\delta'/2$, such that $\diam\xi_{n_j,j}(D(a_0, r_0))\geq S$.
\end{proof}

\subsection{Proof of Theorem \ref{main}}\label{s33}

We have seen that a small disk around the starting Misiurewicz map will finally grow to a definite size under the map $\xi_{n,j}$ due to the expansion on the  post-singular set of the starting map. After this, any given compact set will be covered within finitely many further iterates. The following result tells us this basic fact. See \cite[Lemma 2.2]{baker15}.

\begin{lemma}[Blowup property of Julia sets]\label{blowup}
Let $f$ be a transcendental entire function. Let $U$ be a neighbourhood of $z\in\j(f)$. Then for any compact set $K$ not containing an exceptional point of $f$ there exist $n(K)\in\mathbb{N}$ such that $f^{n}(U)\supset K$ for all $n\geq n(K)$.
\end{lemma}

By an exceptional point it means that the backward orbit of this point is finite. Note that a transcendental entire function has at most one exceptional point.

Let $f$ be a Speiser function as in Theorem \ref{main} and let $f_0$ be Misiurewicz. Per definition, $f_0$ is post-singularly bounded. So there exists $R>0$ such that $\p(f_0)\subset D(0,R)$. Moreover, there are at most finitely many critical points in $D(0,R)$. We put
$$B_{R}:=\left(\bigcup_{c\in\ct(f_0)\cap D(0,R)} D\left(c,\frac{1}{R}\right)\right)\cup\left(\c\setminus \overline{D}\left(0,R\right) \right).$$
We assume that $D(c,1/R)$ is contained in $D(0,R)$. This is, however, not a strong assumption as one can achieve this by considering $D(c,1/(CR))$ instead of $D(c,1/R)$ in the definition of $B_R$ for some large constant $C>0$. By doing this we can also have that disks appearing in $B_R$ have disjoint closures. We will always assume these in the following discussions. Therefore, the complement of $B_R$ is a disk with finitely many small disks removed from its interior; i.e., a bounded multiply connected domain.

One should note that an entire function may have no critical points (i.e., they are locally univalent). Then in this case, being post-singularly bounded without parabolic cycles implies that such functions are Misiurewicz. Examples include, for instance, the exponential family and more generally Nevanlinna functions (i.e., entire functions with polynomial Schwarzian derivatives).

\smallskip
Recall that $\p_H(f_0)$ is defined in \eqref{trunca}. Since singular values in the Julia set of $f_0$ are non-recurrent (at least after finitely many iterates) to critical points in the Julia set and $f_0$ has only attracting cycles (if exist), there exists $R>0$ such that
\begin{equation}\label{R-mis}
\p_H(f_0)\cap B_{R}=\emptyset.
\end{equation}

Let $f\in\s$. Recall that $\m_f$ is the parameter space of $f$. Denote by $\mathit{Mis}_{f}$ the set of all Misiurewicz entire functions in $\m_f$. Now we say that an entire function $g\in\m_f$ is $R$-\emph{Misiurewicz} if there exists $R>0$ such that \eqref{R-mis} holds. Now we define
$$\mathit{Mis}_{f}^{R}:=\left\{g\in\m_f:\, g\text{~is $R$-Misiurewicz}\,\right\}.$$
So if $R_1<R_2$ we have $\mathit{Mis}_{f}^{R_1}\subset \mathit{Mis}_{f}^{R_2}$. So we see that, for $n\in\mathbb{N}$,
$$\mathit{Mis}_{f}=\bigcup_{n>0}\mathit{Mis}_{f}^{n}.$$

 To prove Theorem \ref{main} it suffices to show that for every $n$ the set of $n$-Misiurewicz functions has Lebesgue density less than one at $f$. It is then necessary to show that a definite portion in any parameter disk around the starting  $n$-Misiurewicz function consists of functions which are not $n$-Misiurewicz. We will still use $R$ instead of $n$ in the following discussions.
 
 \smallskip
 Now let $f_0$ be a $R$-Misiurewicz map so that \eqref{R-mis} holds. By Lemma \ref{scalelarge}, for each $j\in\I$ and for all sufficiently small $r>0$ there exist $S>0$ (large scale) and $n_j\in\mathbb{N}$ such that $\xi_{n_j,j}(D(a_0, r_0))$ has diameter at least $S$, where $D(a_0,r_0)\subset D(0,r)$ is $\kappa$-Whitney disk. Moreover, by the bounded distortion in Lemma \ref{est} we have $\xi_{n_j,j}(D(a_0, r_0))\supset D(\xi_{n_j,j}(a_0), S/2)$.

\bigskip
To proceed, we have to distinguish whether $D(\xi_{n_j,j}(a_0), S/2)$ is contained in the Fatou set $\f(f_0)$ of $f_0$ for each $j$. For this purpose, recall that we are only considering singular values in the Julia set of $f_0$; see Section 3.1. Let $\H$ denote the index set of all $v_j(0)$ defined in \eqref{kj}. This set either ranges over all singular values of $f_0$ or omits some singular values which must lie in the attracting basins of $f_0$ by Proposition \ref{fatoucomponent}. Then we have that $\I\subset \H$.

\medskip
First we assume that $D(\xi_{n_j,j}(a_0), S/2)\subset \f(f_0)$ for all $j\in \I$. Note that the Fatou set $\f(f_0)$ consists of only attracting basins and under perturbations these attracting basins are persistent. By making perturbations sufficiently small we can have that $D(\xi_{n_j,j}(a_0), S/4)\subset \f(f_a)$ for all $a\in D(0,r)$. Let $P$ denote the set of parameters $a$ in $D(a_0, r_0)$ with this property. Then for each $a\in P$, the corresponding singular values $v_j(a)$ of $f_a$ for $j\in \I$ fall into attracting cycles. For those indices $i$ in $\H\setminus\I$, they belong to the Julia set of $f_a$ and fall into a hyperbolic set, as for such indices the function $x_i(a)$ is identically equal to zero (and thus $v_i(a)$ have same behaviours as the corresponding $v_i(0)$); see Lemma \ref{indi}. In other words, each $f_a$ with $a\in P$ is a Misiurewicz map unless $\I=\H$ in which case $f_a$ is hyperbolic. If the latter is not happening, then one can find a full dimensional ball in $B(0,r)$ such that each parameter $a$ in this ball has the property that all singular values in the Julia set $\j(f_a)$ satisfy that $x_j(a)\equiv 0$. This contradicts with Lemma \ref{indi}. So there exists a positive constant $0<c<1$ such that
$$\meas\{a\in D(a_0, r_0):\,f_a~\text{is hyperbolic}\,\}\geq c\, \meas D(a_0, r_0).$$

\medskip
Now we assume that $D(\xi_{n_j,j}(a_0), S/2)\subset \f(f_0)$ for some (but not all) $j\in \I$. Denote by $\I'$ the set of such indices. Then for each $j\in\I\setminus\I'$, we see that $D(\xi_{n_j,j}(a_0), S/2)\cap\j(f_0)\neq\emptyset$.

\medskip
 We consider the set
 $$V_{R}:=\overline{B}_{2R}\setminus B_{5R}.$$
 By definition, this is a compact set consisting of finitely many annuli. So Lemma \ref{blowup} tells us that there exists some integer $N$ such that
$$f_{0}^{n}(D(\xi_{n_j,j}(a_0), S/2))\supset V_{R}$$
for all $n\geq N$. By taking perturbations sufficiently small (i.e., $r$ small) we can have that for all $a\in D(a_0, r_0)$,
$$f_{a}^{n}(D(\xi_{n_j,j}(a_0), S/2))\supset V'_{R}$$
for all $n\geq N$, where
$$V'_R=\overline{B}_{3R}\setminus B_{4R}.$$
Since $N$ depends only on $f_0$ and $S$ there exists a positive number $c'>0$ such that
$$\meas f_{a}^{-N}(V'_R)\geq c'\meas D(\xi_{n_j,j}(a_0), S/2),$$
where the inverse branch is chosen appropriately. For simplicity, put $V=f_{a}^{-N}(V'_R)$. So $V\subset  \xi_{n_j,j}(D(a_0, r_0))$. We also get that
$$\meas V\geq c''\meas \xi_{n_j,j}(D(a_0, r_0))$$
for some constant $c''=c''(c')>0$. Now by Lemma \ref{est} there exists a positive number $c>0$ such that
\begin{equation}\label{depo}
\meas \xi_{n_j,j}^{-1}(V)\geq c \meas D(a_0, r_0).
\end{equation}

We claim that parameters in $\xi_{n_j,j}^{-1}(V)$ are not $R$-Misiurewicz. By construction, $\xi_{n_j+N,j}(a)\in V'_R$ for $a\in \xi_{n_j,j}^{-1}(V)$, which means that $f_a$ cannot be $R$-Misiurewicz. Now \eqref{depo} says that a definite portion in $D(0,r)$ do not correspond to $R$-Misiurewicz maps. The theorem is thus proved.

\section{Perturbations of Misiurewicz maps}

One of the central problems in real and complex dynamics is whether hyperbolic dynamics are dense (Fatou's conjecture). This is still open in the complex setting. For transcendental functions, Fatou's conjecture is believed to be true for Speiser functions. More precisely, let $f\in\s$. Assume that $f$ is non-hyperbolic. Can one find hyperbolic maps in $\m_f$ which are arbitrarily close to $f$? In this section, by using results from previous section we show that this is true for Misiurewicz entire functions. 

Approximating Misiurewicz maps is easier if it has only critical values (regardless whether the Fatou set is empty or not) as long as we have Lemma \ref{scalelarge} at hand. However, transcendental functions may not have any critical values. For instance, the map $\lambda e^z$ or more generally Nevanlinna functions (i.e., entire functions with polynomial Schwarzian derivatives). If this happens, a different argument has to be used. To overcome this we shall use function-theoretic behaviours of entire functions around points taking maximum modulus. This is usually called the \emph{Wiman-Valiron theory}. Classically the theory deals with transcendental entire functions; see, for instance, \cite{hayman2}. Recently a generalization was given by Bergweiler, Rippon and Stallard for meromorphic functions with direct or logarithmic tracts \cite{bergweiler3}. We will only use part of their theory for logarithmic tracts.

\subsection{Wiman-Valiron theory and growth in tracts}
We collect here some results in the Wiman-Valiron theory from \cite{bergweiler3}.

\begin{definition}\label{logtract}
Let $T$ be an unbounded, simply connected domain in $\c$ whose boundary consists of piecewise smooth curves. Suppose that $\c\setminus T$ is unbounded. Let $f$ be a complex-valued function with domain of definition containing $\overline{T}$. Then we say that $T$ is a logarithmic tract of $f$ if $f$ is holomorphic in $T$, continuous in $\overline{T}$, and moreover, there exists $R>0$ such that $|f(z)|=R$ for $z\in\partial T$ and $|f(z)|>R$ for $z\in T$, and $f: T\to \{z: |z|>R\}$ is a universal covering.
\end{definition}

The above defines a logarithmic tract over $\infty$ for $f$. Similarly one can define logarithmic tracts over any finite value $s$ by considering $1/(f-s)$.

For us, it is worth mentioning that every asymptotic value of a Speiser function has at least one corresponding logarithmic tract. More precisely, suppose that $f\in\s$ and $s$ is an asymptotic value of $f$. Then by taking a suitable $r>0$, each unbounded component $T$ of $f^{-1}(D(s,r))$ is a logarithmic tract of $f$. We will denote by $T_f(s)$ the logarithmic tract of $f$ corresponding to the singular value $s$.

Using notations in the above Definition \ref{logtract}, put
$$v(z)=
\begin{cases}
~\,\log\dfrac{|f(z)|}{R}~&\text{if}~z\in T,\\
~\,0~&\text{if}~z\not\in T,
\end{cases}
$$
and
\begin{equation}\label{deff}
B(r,v)=\max_{|z|=r} v(z)~\quad\text{and}\quad~a(r,v)=r B'(r,v).
\end{equation}
We refer to \cite[Section 2]{bergweiler3} for more details concerning properties of $a(r,v)$ and $B(r,v)$. We also use $M(r,f)=\max_{|z|=r}|f(z)|$. The following result is part of the Wiman-Valiron theory.

\begin{theorem}\label{brs}
Let $T$ be a logarithmic tract of $f$ and $\tau>1/2$. Let $z_r$ be such that $|z_r|=r$ and $v(z_r)=B(r,v)$. Then there exists a set $F\subset [1,\infty)$ such that $\int_F dr/r<\infty$ and such that if $r\in [1,\infty)\setminus F$, then $D(z_r, r/a(r,v)^\tau)\subset T$. Moreover,
\begin{equation}\label{brs1}
f(z)\sim \left(\frac{z}{z_r}\right)^{a(r,v)}f(z_r)~\quad\text{for}\quad~ z\in D\left(z_r, \frac{r}{a(r,v)^\tau}\right),
\end{equation}
and
\begin{equation}\label{brs2}
f'(z)\sim \frac{a(r,v)}{z}\left(\frac{z}{z_r}\right)^{a(r,v)}f(z_r)~\quad\text{for}\quad~ z\in D\left(z_r, \frac{r}{a(r,v)^\tau}\right).
\end{equation}
\end{theorem}
For simplicity, we call $D(z_r, r/a(r,v)^\tau)$ occurring in the above theorem \emph{Wiman-Valiron disks} (for short, WV-disks). We shall use this result repeatedly to find a (finite) sequence of points $z_n$ in the logarithmic tracts $T_{f}(\infty)$ such that the WV-disk around $z_{n-1}$ will cover $z_{n}$ under $f$. In this way, we want to achieve that $|Df^{n}(z_0)|$ is dominated by the position of $z_n$. When this happens, we turn our attention to the logarithmic tract $T_{f}(s)$ over some finite asymptotic value $s$. Then the WV-disk at $z_n$ will cover some WV-disk around a point $z_{n+1}\in T_f(s)$ under the map $f$. Therefore, $f(z_{n+1})$ will be very close to $s$. Along this procedure we produce a sequence which goes far away from the origin and then return to some asymptotic value in $\c$.

Suppose that \eqref{brs2} in the above lemma holds. Then we see that $|f'(z)|$ for $z$ in the WV-disk around $z_r$ is dominated by $|f(z_r)|$. The term $|a(r,v)/r|$ is small compared with $|f(z_r)|$. This can be seen by \cite[Lemma 6.10]{bergweiler3} which says that
$$a(r,v)\leq B(r,v)^{1+\varepsilon}$$
for some $\varepsilon>0$, except for a set $F\subset [1,\infty)$ such that $\int_F dr/r<\infty$. By \eqref{deff} we also have
$$B(r,v)\leq\log M(r,f).$$
Therefore, combining all the above we have
\begin{equation}\label{avr}
\left|\frac{a(r,v)}{z} \right|\leq \frac{(\log M(r,f))^{1+\varepsilon}}{r},
\end{equation}
except for a set $F\subset [1,\infty)$ such that $\int_F dr/r<\infty$.

A consequence of the above theorem is the following result; see \cite[Theorem 2.3]{bergweiler3}.

\begin{lemma}\label{brs3}
For each $\beta>1$ there exists $\alpha>0$ such that if $f, T, v, z_r$ and $F$ are as in the above theorem and if $r\not\in F$ is sufficiently large, then
$$f\left(D\left(z_r, \frac{\alpha r}{a(r,v)}\right) \right)\supset A\left(\frac{|f(z_r)|}{\beta}, \beta|f(z_r)| \right).$$
\end{lemma}

\bigskip
To show that a Misiurewicz map with empty Fatou set can be perturbed to another Misiurewicz map with non-empty Fatou set, we will also need some information about the growth of the maximum modulus inside logarithmic tracts. This is essentially the well known Denjoy-Carleman-Ahlfors theorem \cite{goldbergmero}. We state it as a lemma for our convenience; see, for instance, the discussion in \cite[Section 1]{aspenberg1}.

\begin{lemma}\label{growintract}
Let $T$ be a logarithmic tract of $f$ (over $\infty$). Put $M_{T}(r,f)=\max_{z\in T}|f(z)|$. Then there exist $r_0>0$ and $C>0$ such that
$$\log\log M_T(r,f)\geq \frac{1}{2}\log r - C$$
for $r\geq r_0$.
\end{lemma}

\subsection{An inducing scheme}

We are now in a position to prove Theorem \ref{mishyp}. The proof is constructive by using an inducing scheme. Roughly speaking, our starting Misiurewicz map will induce another nearby Misiurewicz map by pushing one singular value from the Julia set to the Fatou set. Then we start with this new Misiurewicz map and conduct same analysis as before to get another Misiurewicz map. Inductively we will finally push all singular values to the Fatou set and thus get a hyperbolic map.

\medskip
Let $f$ be a Misiurewicz map. We consider the function 
$$\xi_{n,j}(a)=f_{a}^{n}(v_j(a))$$
as defined in \eqref{pp} for each $j$ around our starting Misiurewicz map $f_0:=f$. Then by Lemma \ref{scalelarge}, for each $j$ there exists $n_j\in\mathbb{N}$ such that $\xi_{n_j,j}(D(a_0, r_0))$ reaches large scale $S$ for all $\kappa$-Whitney disks. Recall that $D(a_0,r_0)\subset D(0,r)$.

\smallskip
We first prove the following result, which is Theorem \ref{mishyp} under the additional condition that the Fatou set is not empty.

\begin{theorem}\label{thm42}
Let $f\in\s$ be a Misiurewicz entire function for which $\j(f)\neq\c$. Then there are hyperbolic maps in $\m_f$ arbitrarily close to $f$.
\end{theorem}

\begin{remark}
Stronger result actually holds in this case. For consistency we leave it to the next section.
\end{remark}

\begin{proof}[Proof of Theorem \ref{thm42}]

As before, we consider a small parameter ball $B(0,r)$ centred at $f_0:=f$ in the parameter space $\m_f$. By Lemma \ref{scalelarge}, for each singular value $s_j(0)$ in the Julia set of $f_0$ there exist $n_j$ and a number $0<\kappa<1$ such that the diameter of $\diam\xi_{n_j,j}(D(a_0,r_0))$ is at least $S$ for any $\kappa$-Whitney disk $D(a_0, r_0) \subset D(0,r) \subset B(0,r)$.

Let $\H$ be the set of indices ranging over singular values in the Julia set $\j(f_0)$, and $\I$ the index set for which $x_j(a) \not\equiv 0$; see \eqref{xj}, Lemma \ref{indi} and the discussion thereafter.

We first focus on indices in $\I$ and consider $\xi_{n_j, j}(D(a_0, r_0))=:D_j$ for $j\in\I$. By the bounded distortion ensured in Lemma \ref{est}, we see that $D_j$ contains a disk of radius $S/4$. We shall consider two situations depending on whether $D_j$ intersects with the Julia set $\j(f_{0})$.

\medskip

If $D_j\subset \f(f_0)$ holds for all $j\in \I$, then by considering $D'_j=\xi_{n_j,j}(D(a_0, r_0/4))$ if necessary, we can achieve the following property: Each $a\in D(a_0, r_0/4)$ satisfies that $v_j(a)\in\f(f_a)$. However, for all indices in $\H\setminus\I$ we still have that $v_j(a)$ falls into some hyperbolic set under iterates of $f_a$ as transversality fails for these singular values. So we see that each parameter in $D(a_0, r_0/4)$ is a Misiurewicz map if $\H\setminus\I\neq\emptyset$. But this means that one can find a component of quasiconformally conjugate Misiurewicz maps  in the full dimensional ball $B(0,r)$, which contradicts with the fact that Misiurewicz maps do not carry invariant line fields. Hence, $\H\setminus\I=\emptyset$. Therefore, every parameter in $D(a_0, r_0/4)$ is a hyperbolic map.

\smallskip

Now suppose that there is a subset $\I'$ of $\I$ such that $D_j\cap\j(f_0)\neq\emptyset$ for $j\in\I'$. Without loss of generality, assume that $\I'=\{1,\dots,m\}$ with $m=\#\I'$. Note that each set $\mathcal{A}_j:=\{a\in B(0,r): x_j(a)=0\}$ is an analytic set of codimension one for $j\in\I$. Then the intersection $\mathcal{A}:=\cap_{j\in\I'\setminus\{m\}}\mathcal{A}_j$ is an analytic set of codimension $m-1$. Since $x_j$ is identically equal to zero for $j\in\H\setminus\I$, this means that $x_m$ is not identically equal to zero in $\mathcal{A}$. Assume that the Whitney disk $D(a_0, r_0)\subset B(0,r)\cap \mathcal{A}$. Since Julia sets are nowhere dense, one can always find a nonempty compact subset $K \subset D_m \setminus \j(f_0)$. Hence $K \subset \f(f_0)$ and by making $r > 0$ sufficiently small, we see that this compact $K \subset \f(f_a)$, for all $a \in B(0,r)$. By compactness we conclude that there exists an $r > 0$ so that for every such disk $D_m$, there is a non-empty compact $K \subset D_m \setminus \j(f_0)$ for which $ K \subset \f(f_a)$, for $a \in B(0,r)$. So, one can achieve that for each $a\in \xi_{n_m, m}^{-1}(K)\cap \mathcal{A}$ we have $v_m(a)\in\f(f_a)$. Compared with the starting Misiurewicz map $f_0$, the map $f_a$ with $a\in \xi_{n_m, m}^{-1}(K)\cap \mathcal{A}$ has one less singular value in the Julia set $\j(f_a)$. Now starting with this new Misiurewicz map $f_a$, one can conduct the same analysis as above to "push" $v_j(a)$ with $j\in \I'$ to the Fatou sets. Finally we obtain a Misiurewicz map whose singular values $v_j(a)$ with $j\in\I$ are attracted by attracting basins. However, since $x_j$ is identically equal to zero for $j\in\H\setminus\I$, we get a contradiction similarly as above.

As the parameter ball $B(0,r)$ can be taken arbitrarily small, we see that the obtained hyperbolic map can be arbitrarily close to $f_0$.
\end{proof}

Now we deal with those functions with empty Fatou set.
\begin{theorem}\label{thm43}
Let $f\in\s$ be a Misiurewicz entire function for which $\j(f)=\c$. Then there is a Misiurewicz map $g\in\m_f$ arbitrarily close to $f$ for which $\j(g)\neq\c$.
\end{theorem}

\begin{proof}
Depending on whether $f$ has critical values or not, we divide the proof into two cases. Again, let $\H$ be the set of indices of singular values in the Julia set of $f_0$ and $\I$ the set of indices $j$ for which $x_j$ is transversal. With $p:=\#\I$ we assume, without loss of generality, that $\I=\{1,\dots,p\}$. We shall use the argument similarly as above. Since $\mathcal{A}_j:=\{a\in B(0,r): x_j(a)=0\}$ is an analytic set of codimension one for $j\in\I$. We see that $\mathcal{A}:=\cap_{j\in\I\setminus\{p\}}\mathcal{A}_j$ is an analytic set of codimension $p-1$. So $x_m$ is not identically equal to zero in $\mathcal{A}$. Take a Whitney disk $D(a_0, r_0)\in B(0, r)\cap\mathcal{A}$. By Lemma \ref{scalelarge} one can find $n_p$ such that diameter of $\xi_{n_p,p}(D(a_0,r_0))=:D_p$ is at least $S$. Lemma \ref{est} tells that $D_p$ contains a disk of radius $S/4$, denoted by $D'_p$.

\bigskip
In the first case we assume that $f_0=f$ has at least one critical value. We may assume that this critical value has index in $\I$. If not, we focus on any one of singular values with indices in $\I$. In case that all these singular values are asymptotic value, we use Wiman-Valiron theory in the same way as in the second case. So, take one of them and denote it by $s_j(0)$ with $j\in\I$. Without of loss generality, we may assume that $j=p \in \I$. So we are considering $s_p(0)$ (see Section 3). We fix the rest as before putting $\mathcal{A}:=\cap_{j\in\I\setminus\{p\}}\mathcal{A}_j$. Since $s_p(a)$ is transversal in $\mathcal{A} \cap B(0,r)$ there is $n_p$ and a subdisk $D(a_0,r_0) \subset B(0,r)$ such that $\xi_{n_p, p}(D(a_0, r_0))$ has reached the large scale and contains a disk of radius $S/4$. Since $\j(f_0)=\c$, we obtain using Theorem \ref{blowup} that there exists $N_p$ such that $f_{0}^{N_p}(\xi_{n_p, p}(D(a_0, r_0)))\supset D(c_p(0), \delta)$ where $c_p(0)$ is a critical point of $f_0$ such that $f_0(c_p(0))=s_p(0)$, and $\delta>0$ is some number. Now by Brouwer's Fixed Point Theorem, one can see that there exists a parameter $a\in D(a_0, r_0)$ such that $f_{a}^{N_{p}}(s_p(a))=c_p(a)$. In other words, $s_p(a)$ falls into a super-attracting cycle of $f_a$. All other singular values $s_k(a)$ of $f_a$ stay in the Julia set of $f_a$ and have the same behaviours to the corresponding $s_k(0)$. So $f_a$ is a Misiurewicz entire function but has non-empty Fatou set.

\medskip
In the second case we assume that $f_0=f$ has at least one asymptotic value in $\I$.
%That is to say, all singular values in $\I$ of $f_0$ are asymptotic values.
To perturb $f_0$ to a Misiurewicz map in $\m_f$ with non-empty Fatou set, we will use the above Wiman-Valiron theory.

First we recall that $R$ is defined in \eqref{defR} such that the post-singular set of $f$ is contained in $D(0,R)$. Each (unbounded) component of $f^{-1}(\c\setminus\overline{D}(0,R))$ is a logarithmic tract of $f$ over $\infty$. (Note that $f$ in general may have infinitely many logarithmic tracts over $\infty$. Consider for example the function $e^{e^z}$.) For our purpose, it suffices to consider only one of them. Denote the tract we are focusing on by $T_{f}(\infty)$. Then by Theorem \ref{brs}, there is an exceptional set $F$ of radii with finite logarithmic measure such that the Wiman-Valiron theory holds in $T_{f}(\infty)$. For simplicity, we will also use, for $x>0$,
$$E(x)=\exp\sqrt{x}.$$
The $n$-th iterate of $E$ is denoted by $E^{n}$.

By choosing $r_1>0$ large enough and $r_1\not\in F$ and $\beta=2$ in Lemma \ref{brs3} we see that there exists $z_1\in T_{f}(\infty)$ such that with $|z_1|=r_1$ 
\begin{equation}
f\left(D\left(z_1, \frac{\alpha r_1}{a(r_1,v)}\right) \right)\supset A\left(\frac{|f(z_1)|}{2}, 2|f(z_1)| \right).
\end{equation}
Since $r_1$ is large we can have that $[|f(z_1)|, 2|f(z_1)|]\setminus F\neq\emptyset$ and moreover, by Lemma \ref{growintract}, $|f(z_1)|>E(|z_1|)=E(r_1)$. So we can choose a point $z_2\in T_{f}(\infty)$ such that $|z_2|=r_2 \in [|f(z_1)|, 2|f(z_1)|]\setminus F$ such that
\begin{equation}
f\left(D\left(z_2, \frac{\alpha r_2}{a(r_2,v)}\right) \right)\supset A\left(\frac{|f(z_2)|}{2}, 2|f(z_2)| \right)
\end{equation}
and
$$|f(z_2)|>E(r_2)=E(|z_2|)\geq E(|f(z_1)|)\geq E^{2}(|z_1|).$$
Continue inductively we can find a sequence of points $z_j\in T_{f}(\infty)$ for $1\leq j\leq n$ satisfying $|z_j|=r_j\in [|f(z_{j-1})|, 2|f(z_{j-1})|]\setminus F$ such that
\begin{equation}\label{nind}
f\left(D\left(z_{j}, \frac{\alpha r_j}{a(r_j,v)}\right) \right)\supset A\left(\frac{|f(z_j)|}{2}, 2|f(z_j)| \right)
\end{equation}
 and
\begin{equation}\label{increase}
|f(z_j)|>E(|z_j|)\geq E(|f(z_{j-1})|)\geq E^2(|z_{j-1}|)\geq\cdots\geq E^{j}(|z_1|)
\end{equation}
for $1\leq j\leq n-1$. Moreover, by \eqref{brs2} in Theorem \ref{brs} we see that 
\begin{equation}
f'(z)\sim \frac{a(r_j,v)}{z}\left(\frac{z}{z_j}\right)^{a(r_j,v)}f(z_j)~\quad\text{for}\quad~ z\in D\left(z_j, \frac{\alpha r_j}{a(r_j,v)}\right).
\end{equation}
This, combined with \eqref{avr}, gives the estimate
\begin{equation}\label{rouest}
|f'(z)|\leq |f(z_j)|^2 ~\quad\text{for}\quad~ z\in D\left(z_j, \frac{\alpha r_j}{a(r_j,v)}\right).
\end{equation}
We note that the choice of $n$ depends on the discussion below.

\begin{figure}[htbp] %  figure placement: here, top, bottom, or page
	\centering
	\includegraphics[width=10.3cm]{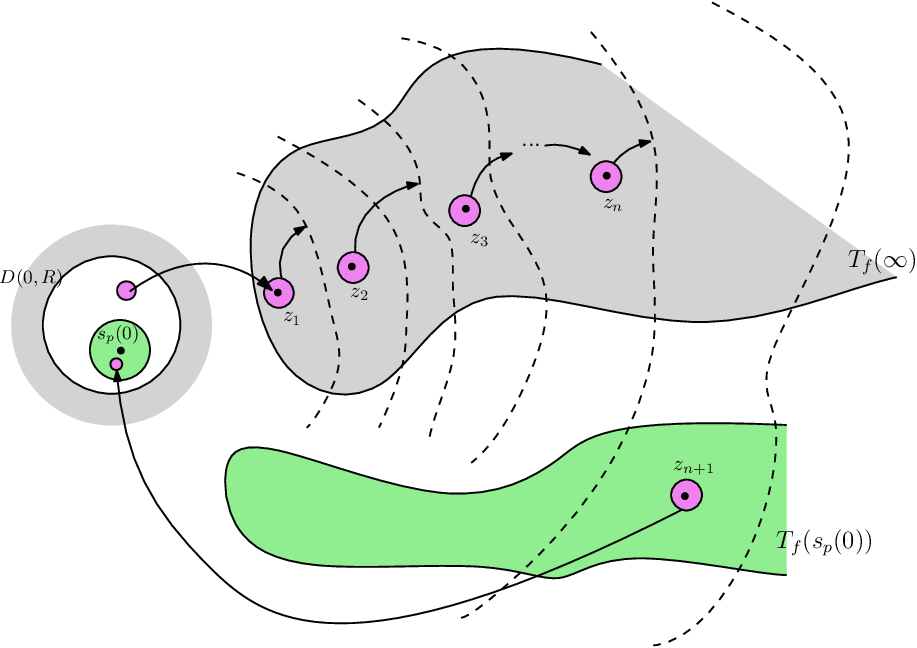}
	\caption{Constructing Misiurewicz maps with non-empty Fatou sets around a Misiurewicz map which has only asymptotic values. The Wiman-Valiron theory is used to create a sufficiently long iterate in logarithmic tracts which "swallows" all previous derivative growth.}
	\label{pic}
\end{figure}

As in the first case, we consider the singular value $s_p(0)$ of $f_0$. Note that by assumption, $s_p(0)$ is an asymptotic value. Let $T_{f}(s_p(0))$ be one of the logarithmic tracts over $s_p(0)$ such that $f(T_{f}(s_p(0)))=D(s_p(0), r')\setminus\{s_p(0)\}$ for some $r'>0$. By \eqref{nind} and by taking a sufficiently large $n$ we see that
\begin{equation}
A\left(\frac{|f(z_n)|}{2}, 2|f(z_n)|\right) \cap T_{f}(s_p(0)) \neq\emptyset.
\end{equation}
Moreover, if $F_j$ is the exceptional set corresponding to the logarithmic tract $T_{f}(s_p(0))$, one can make sure by taking $n$ large that $[|f(z_n)|/2, 2|f(z_n)|]\setminus F_j\neq\emptyset$. Now instead of considering $f$ itself, we consider the function
$$\tilde{f}(z)=\frac{1}{f(z)-s_p(0)}$$
for $z\in T_f(s_p(0))$. For this function (we are not assuming that $\tilde{f}$ is globally defined), $T_f(s_p(0))$ is a logarithmic tract of $\tilde{f}$ over $\infty$. Thus one can apply the Wiman-Valiron theory in $T_f(s_p(0))$ to $\tilde{f}$ and similarly define the function $\tilde{v},\,\tilde{a}(r,\tilde{v})$ corresponding to $v$ and $a(r,v)$ respectively as before.

Therefore, there exists a point $z_{n+1}\in T_{f}(s_p(0))$ such that
\begin{equation}\label{n1ind}
f\left(D\left(z_{n}, \frac{\alpha r_n}{a(r_n,v)}\right) \right)\supset D\left(z_{n+1}, \frac{\tilde{\alpha} r_{n+1}}{\tilde{a}(r_{n+1},\tilde{v})}\right)
\end{equation}
 and
\begin{equation}\
D\left(z_{n+1}, \frac{\tilde{\alpha} r_{n+1}}{\tilde{a}(r_{n+1},\tilde{v})}\right)\subset T_{f}\left(s_p(0)\right).
\end{equation}
Moreover,
$$|\tilde{f}(z_{n+1})|>E(|z_{n+1}|).$$
This gives us, returning to the original function $f$, that
\begin{equation}\label{n1e}
|f(z_{n+1})-s_p(0)|<\frac{1}{E(|z_{n+1}|)}.
\end{equation}

\smallskip
Put $f_0=f$. Now by Lemma \ref{scalelarge}, there exists $n_p$ such that  $D_0:=\xi_{n_p,p}(D(a_0, r_0))$ reaches large scale $S$ as before, where $D(a_0, r_0)$ is any Whitney disk contained in $\mathcal{A} \cap B(0,r)$ and $\mathcal{A} =\cap_{j=1}^{p-1}\mathcal{A}_j$. Our purpose is to find a Misiurewicz map in this disk with non-empty Fatou set. 

Since $\j(f_0)=\c$, Theorem \ref{blowup} implies that there exists $M\in\mathbb{N}$ such that 
$$f_{0}^{M}(D_0)\supset D\left(z_1, \frac{\alpha r_1}{a(r_1,v)}\right).$$
Then by the above construction of $z_j$ we see that there exists a set $D_1\subset D_0$ such that
$$f_{0}^{j-1}(z)\in D\left(z_j, \frac{\alpha r_j}{a(r_j,v)}\right)~\quad\text{for}\quad~1\leq j\leq n+1~\text{~and~}~ z\in f_{0}^{M}(D_1).$$
By considering a longer sequence $(z_j)$ if necessary (i.e., taking a larger $n$), we can achieve that
\begin{equation}\label{max}
|z_{n+1}|\geq\sup_{z\in D_1, 0\leq j\leq M+n} |f_{0}^{j}(z)|.
\end{equation}
Points in $D_1$ have the following behaviours: they first ``escape'' to $\infty$ under long iterates and then come back very close to the singular value $s_j(0)$.

Define
$$D:=\xi_{n_p, p}^{-1}(D_1).$$
This is a set of parameters consisting of maps whose singular value $s_p(a)$ has the above behaviour while all the rest of singular values have same behaviours as to the corresponding singular values for $f_0$. Now we claim that there is a Misiurewicz parameter in $D$ with non-empty Fatou sets.

\smallskip

To see this, with $D_{n+1}:=D\left(z_{n+1}, \frac{\tilde{\alpha} r_{n+1}}{\tilde{a}(r_{n+1},\tilde{v})}\right)$ we consider the pullback of $D_{n+1}$ under the map $f_{a}^{k+n_p+M+n}$, where $a \in D$. (Recall that $k$ is defined in \eqref{kj}.) This is a domain containing the singular value $s_p(a)$. By the Koebe distortion, we see that
$$f_{a}^{-(k+n_p+M+n)}(D_{n+1})\supset D(s_p(a), r_1)$$
with 
$$r_1\geq \frac{c}{|Df_{a}^{k+n_p+M+n}(s_p(a))|}\frac{\tilde{\alpha} r_{n+1}}{\tilde{a}(r_{n+1},\tilde{v})}$$
for  some constant $c>0$. We also have
$$f_a(D_{n+1})\subset D(s_p(a), r_2)$$
with
\begin{equation}\label{r2e}
r_2\leq \frac{c'}{E(|z_{n+1}|)}
\end{equation}
by \eqref{n1e} for some constant $c'>0$.
Now we shall compare $r_1$ and $r_2$. Since, by the chain rule,
$$|Df_{a}^{k+n_p+M+n}(s_p(a))|=|Df_{a}^{n}(f_{a}^{k+n_p+M}(s_p(a))|\cdot|Df_{a}^{k+n_p+M}(s_p(a))|$$
and the latter is small compared with the former by taking $n$ sufficiently large, we have
$$|Df_{a}^{k+n_p+M+n}(s_p(a))|\leq c_1 |Df_{a}^{n}(f_{a}^{k+n_p+M}(s_p(a))|$$
for some constant $c_1>0$. Put $w=f_{a}^{k+n_p+M}(s_p(a))$. Then by the construction,  \eqref{brs2}, \eqref{max} and \eqref{rouest} we see that for some constant $c_2\geq 1$
$$|Df^{n}_{a}(w)|=\prod_{j=0}^{n-1}|f'_{a}(f_{a}^{j}(w))|\leq |f_{a}^{n}(w)|^2\leq c_2 |z_{n+1}|^2.$$
So we see that using \eqref{avr} and \eqref{r2e}
$$r_1\geq \frac{c}{c_1 c_2 |z_{n+1}|^2}\frac{\tilde{\alpha} r_{n+1}}{\tilde{a}(r_{n+1},\tilde{v})}\geq \frac{c_3}{|z_{n+1}|^{1+\delta}}>r_2$$
for some constant $c_3>0$ and $\delta>0$.
In other words, we have that
$$\xi_{a}^{k+n_p+M+n}\left(\overline{D}\left(s_p(a), r_1\right)\right)\subset D(s_p(a), r_2)\subset D(s_p(a), r_1).$$
This means that there exists a parameter in $a\in D$ which has an attracting periodic point such that the singular value $s_p(a)$ lies in some attracting basin. As we leave all other singular values unchanged for $f_a$, we see that $f_a$ is a Misiurewicz map for which $\j(f_a)\neq\c$.
\end{proof}

\begin{proof}[Proof of Theorem \ref{mishyp}]
Let $f\in\s$ be Misiurewicz. We consider two situations. 

\smallskip
$(i)$ If $\j(f)\neq\c$, then Theorem \ref{thm42} gives the desired approximation of $f$ by some hyperbolic map in $\m_f$. As the full dimensional ball $B(0,r)$ can be arbitrarily small, we see that there are always hyperbolic maps arbitrarily close to $f$.

\smallskip
$(ii)$ Suppose now that $\j(f)=\c$. By Theorem \ref{thm43}, one can find another Misiurewicz map $g\in\m_f$ for which $\j(g)\neq\c$. Then  the above case $(i)$ can be applied to this new Misiurewicz map $g$ and thus we can reach the same conclusion.

This completes the proof of Theorem \ref{mishyp}.
\end{proof}

\section{Misiurewicz and hyperbolicity}

In this section we prove the following result, which  contains Theorem \ref{denmis} as a special case.

\begin{theorem}\label{denmis2}
Let $f\in\s$ be Misiurewicz for which $\f(f)\neq\emptyset$. Then the set of hyperbolic maps in $\m_f$ has positive density at $f$. If, in addition, $\j(f)$ has zero Lebesgue measure, $f$ is a Lebesgue density point of hyperbolic maps in $\m_f$.
\end{theorem}

Let $f_0:=f\in\s$ be such a Misiurewicz function. As before we consider a parameter ball $B(0,r)$ centered at $f_0$ of radius $r$, and consider the function $\xi_{n,j}$ defined in $B(0,r)$; cf. \eqref{pp}. Since $\f(f_0)\neq\emptyset$,  we know that $\f(f_0)$ has only attracting cycles by Proposition \ref{fatoucomponent}. Let $U\subset\f(f_0)$ be a neighbourhood of attracting cycles of $f_0$. Let $U_{\delta}$ be a $\delta$-neighbourhood of $\j(f_0)$; that is to say,
$$U_{\delta}=\left\{z:\,\dist(z,\j(f_0))<\delta   \right\}.$$
Then
$$V_{\delta}:=\overline{D}(0, R)\setminus U_{\delta}$$
is a compact subset of $\f(f_0)$. Now if $W$ is a neighbourhood of attracting periodic cycles of $f_0$, there exists $n$ such that 
$$V_{\delta}\subset \left(\bigcup_{j\leq n}f_{0}^{-j}(W)\right)\bigcap \overline{D}(0,R).$$
By choosing $r$ sufficiently small we see that $V_{\delta}\subset\f(f_a)\cap D(0,R)$ for all $a\in B(0,r)$. Since $\j(f_0)$ is nowhere dense, there is a positive number $c>0$ such that for any disk $D\subset D(0,R)$ of radius $\geq S/4$
\begin{equation}\label{vmea}
\frac{\meas (V_{\delta}\cap D)}{\meas D}\geq c>0.
\end{equation}
Here we are not claiming that $c\to 1$ when $\delta\to 0$, since Julia sets of Misiurewicz entire functions can possibly have positive Lebesgue measure (even if the Fatou set is not empty). We refer to the next section for a discussion on this respect.

By Lemma \ref{scalelarge}, for each $j\in\I$ (where $\I$ was defined in \eqref{defx}) there exists $n_j$ such that $\xi_{n_j, j}(D(a_0, r_0))$ has diameter at least $S$. By the distortion in Lemma \ref{est}, $\xi_{n_j,j}(D(a_0, r_0))$ contains a disk of radius at least $S/4$. Put
$$\H_{\delta,j}:= \xi_{n_j, j}^{-1}\left(\xi_{n_j,j}(D(a_0, r_0))\cap V_{\delta}\right).$$
Again by Lemma \ref{est} and \eqref{vmea} we have
\begin{equation}\label{lebdens}
\frac{\meas \H_{\delta,j}}{\meas D(a_0, r_0)}\geq c'>0,
\end{equation}
where $c'$ depends only on $c$.

With
$$\H_{\delta}:=\bigcap_{j\in\I}\H_{\delta,j}$$
we argue similarly as in Section \ref{s33} that each parameter $a\in\H_{\delta}$ is a hyperbolic parameter. We omit details here. So the Lebesgue density of hyperbolic parameters at the Misiurewicz map $f_0$ is strictly positive.

Now if the Julia set of  $f$ has zero Lebesgue measure, one can achieve that $c\to 1$ as $\delta\to 0$ in \eqref{vmea}. So this means that $c'\to 1$ as $r_0\to 0$ in \eqref{lebdens}. Now by the Lebesgue density theorem and Fubini's theorem, hyperbolic maps are Lebesgue dense at Misiurewicz parameters in this case.

\bigskip
\begin{remark}
One might wonder if a similar result as in Theorem \ref{denmis2} holds for Misiurewicz entire functions with empty Fatou sets. This, however, would depend on the geometry of the functions in question. For instance, Dobbs proved that Misiurewicz maps in the exponential family (whose Julia sets are the whole plane) are Lebesgue density points of hyperbolic maps \cite{dobbs1}. In other families of transcendental entire functions, one could expect a different phenomenon  similar as rational maps: they are (Lebesgue) density points of non-hyperbolic maps with certain expansion property (such as ``Collet-Eckmann''); see, for instance, \cite{abc1}.
\end{remark}

\section{Lebesgue measure of Misiurewicz Julia sets}

We prove in this section a result on the Lebesgue measure of Misiurewicz Julia sets. Several examples will also be given showing that such Julia sets can have zero or positive measure when the Fatou sets are not empty.

For rational Misiurewicz functions, the Julia set is either the whole Riemann sphere or of zero Lebesgue measure (actually even Hausdorff dimension less than two). In the transcendental setting, it is possible that $\meas\j(f)>0$ when $\j(f)\neq\c$, as shown by examples later on. On the other hand, there are indeed some conditions ensuring that $\meas\j(f)=0$ for Misiurewicz functions.

Recall that the order of an entire function $f$ is defined as
$$\rho(f)=\limsup_{r\to\infty}\frac{\log\log M(r,f)}{\log r},$$
where $M(r,f)=\max_{|z|=r}|f(z)|$ is the maximum modulus. We say that $f$ has finite order if $\rho(f)<\infty$. We also recall that $a$ is an asymptotic value of an entire function $f$ if there exists a curve $\gamma$ tending to $\infty$ such that $f(\gamma)\to a$. Since we are only considering Speiser functions here, every asymptotic value of a Speiser function is of logarithmic type.

For an entire function $f$ the escaping set $\I(f)$ is the set of points tending to $\infty$ under iterates of $f$. The following result is due to Eremenko and Lyubich \cite{eremenko2}.

\begin{lemma}\label{areaescape}
Let $f\in\b$ be transcendental and entire. Assume that $\rho(f)<\infty$ and the inverse of $f$ has a finite logarithmic singularity. Then $\meas\I(f)=0$. Moreover, there exists $M>0$ such that 
$$\liminf_{n\to\infty}|f^{n}(z)|<M$$
for almost all $z\in\c$.
\end{lemma}

\begin{theorem}\label{area}
Let $f\in\s$ be Misiurewicz of finite order and $\j(f)\neq\c$. Assume that $f$ has a finite asymptotic value. Then $\area\j(f)=0$.
\end{theorem}

\begin{remark}
Let $f$ be as in the above theorem. It follows from a result of Bara\'nski \cite{baranski1} and Schubert \cite{schubert} that  the Hausdorff dimension of $\j(f)$ is two.
\end{remark}

\begin{proof}
As a hyperbolic set always has zero Lebesgue measure, we see that the post-singular set of $f$ has measure zero. By Lemma \ref{areaescape}, under the assumption of our theorem there exists $M>0$ such that
\begin{equation}\label{bmf}
\liminf_{n\to\infty}|f^{n}(z)|<M
\end{equation}
for almost every $z\in\c$. Take a point $z\in\j(f)$ satisfying \eqref{bmf} for which $f^{n}(z)$ does not belong to the post-singular set of $f$ for all $n$. Then there exist a number $\delta>0$, a sequence $m_j$ and a point $w\in\j(f)$ such that $z_{m_j}:=f^{m_j}(z)\to w$ and $\dist(z_{m_j}, \p(f))>2\delta.$ 

Define $\widetilde{M}:=\max\{M, 2|w|\}$. Since $\j(f)\neq\c$, it is a nowhere dense set and thus by compactness
\begin{equation}\label{skd}
\inf_{\xi:\, |\xi|<\widetilde{M}}\frac{\meas D(\xi,\delta)\cap\f(f)}{\meas D(\xi,\delta)}\geq \varepsilon>0.
\end{equation}
By the choice of $z$ we see that any inverse branch of $f^{m_j}$ can be defined. Using the Koebe distortion theorem we obtain by applying \eqref{skd} to $\xi=z_{m_j}$ that
\begin{equation}\label{skd2}
\frac{\meas f^{-m_j}(D(z_{m_j},\delta))\cap\f(f)}{\meas f^{-m_j}(D(z_{m_j},\delta))}\geq C\varepsilon>0,
\end{equation}
where $C>0$ is some constant independent of $z$ and $\delta$. Koebe distortion also implies that there exist $r_1, r_2>0$ satisfying $r_1\sim r_2$ such that
$$D(z,r_1)\subset f^{-m_j}(D(z_{m_j},\delta)) \subset D(z,r_2).$$
Since $w\in\j(f) \setminus \p(f)$, we have that the derivative of the inverse branch of $f^{m_j}$ tends to $0$ uniformly in $D(w,C'\delta)$ for some constant $C'<1$.

This implies immediately that $z$ is not a Lebesgue density point of $\j(f)$. As $z$ can be chosen arbitrarily (except for a zero measure set), we have the conclusion  by the Lebesgue density theorem.
\end{proof}

Conditions in the above theorem are stated mainly for simplicity. Eremenko and Lyubich gave more general conditions ensuring that the escaping set of an entire function in the class $\b$ has zero Lebesgue measure. See \cite[Theorem 7]{eremenko2} and also a generalization in \cite{cuiwei1}. It is plausible that under these conditions Misiurewicz entire functions also have zero Lebesgue measure Julia sets. On the other hand, conditions implying Julia sets of positive Lebesgue measure also exist; see, for instance, \cite{aspenberg1}.

The rest of this section is devoted to providing several examples of Misiurewicz entire functions with non-empty Fatou set but with positive or zero Lebesgue measure Julia sets.

\begin{example}
$f_{\lambda}(z)=\lambda z^2 e^{z}$, where $\lambda\approx 1.0288$. 
\end{example}

It is easy to see that $f_{\lambda}\in\s$ and $\sing(f_{\lambda}^{-1})=\{0, 4\lambda/e^2\}$. More precisely, $0$ is a critical value and also an asymptotic value at the same time, and $4\lambda/e^2$ is a critical value with the corresponding critical point at $-2$. Moreover, $0$ is a super-attracting fixed point of $f_{\lambda}$ which means that $\j(f_{\lambda})\neq\c$. That $\lambda\approx 1.0288$ is obtained by solving the equation $f_{\lambda}(4\lambda/e^2)=4\lambda/e^2$. With this  chosen $\lambda$ the critical value $4\lambda/e^2$ is a repelling fixed point. So we see that the map $f_{\lambda}$ is a Misiurewicz-Thurston map. In addition, the order of $f_{\lambda}$ is one. So this function indeed satisfies all assumptions in Theorem \ref{area}. Therefore, $\meas\j(f_{\lambda})=0$.

To obtain more examples, one can consider the family 
$$f_{\lambda, m}(z)=\lambda z^m e^z,$$
where $m\geq 2$ and $\lambda\in\c\setminus\{0\}$. Since functions in this family always have a super-attracting fixed point at the origin, the Fatou sets are always non-empty. For each $m$, by choosing appropriate parameters $\lambda$ in a similar way as above, one can obtain Misiurewicz-Thurston maps which satisfy conditions in Theorem \ref{area} and thus have Misiurewicz Julia sets of zero Lebesgue measure.

\smallskip
Perhaps a similar but simpler example is the following.
\begin{example}
$f(z)= z e^{z+1}$. 
\end{example}

One can easily check that $f\in\s$ has two singular values at $0$ and $-1$. Moreover, $0$ is a repelling fixed point of $f$ and $-1$ is a super-attracting fixed point. So $\f(f)\neq\emptyset$ and $f$ is Misiurewicz-Thurston. It also satisfies Theorem \ref{area} and thus has zero Lebesgue measure Julia set.

\smallskip
\begin{example}
Let $g_{\lambda}(z)=\lambda\sin(z^2)$, where $\lambda\approx 3.85299i$. 
\end{example}

By computation, $g_{\lambda}\in \s$ with $\sing(g_{\lambda}^{-1})=\{0,\,\pm\lambda\}$. All singular values of $g_{\lambda}$ are critical values. $0$ is also a critical point. So for any parameter $\lambda$, the function $g_{\lambda}$ has a super-attracting cycle and thus $\j(f)\neq\c$. $\lambda$ is taken as one of the solutions of the equation $z^2\sin(z^2)^2=z^2+2\pi$. Then we see that $g_{\lambda}^2(\lambda)=g_{\lambda}(\lambda)$ but $g_{\lambda}(\lambda)\neq\lambda$. This means that the critical values $\pm\lambda$ are strictly pre-periodic. So $g_{\lambda}$ is Misiurewicz-Thurston. A result of Bock tells that the Julia set of $g_{\lambda}$ has positive measure \cite{bock1}; see also \cite{aspenberg1}. Therefore, the Julia set of $g_{\lambda}$ is not the whole plane but has  positive Lebesgue measure.

\section*{Acknowledgement}
We thank Walter Bergweiler for communicating results on Wiman-Valiron theory. We also thank the referee for a careful reading and valuable suggestions. The second author gratefully acknowledges support from Vergstiftelsen and Qilu Young Scholar Program of Shandong University.

%\bibliography{dynamics}

\newcommand{\etalchar}[1]{$^{#1}$}

\bigskip

\noindent {\bf Magnus Aspenberg}\\
Centre for Mathematical Sciences\\
Lund University, Box 118, 22 100 Lund, Sweden
 
\smallskip
\noindent{magnus.aspenberg@math.lth.se}

\bigskip

\noindent {\bf Weiwei Cui}\\
 Research Center for Mathematics and Interdisciplinary Sciences\\
Frontiers Science Center for Nonlinear Expectations, Ministry of Education\\
Shandong University, Qingdao, 266237, China.

\smallskip

\noindent{weiwei.cui@sdu.edu.cn}

\end{document}